\documentclass[fullpage, reqno]{amsart}

\usepackage{graphicx} % Required for inserting images
 \usepackage{xcolor}
\usepackage{latexsym}
\usepackage{amsmath}
\usepackage{amsfonts}
\usepackage{amssymb}
\usepackage{tikz}
\usepackage{amsthm}
\usetikzlibrary{arrows, shapes}
\usepackage{graphicx}
\usepackage{enumitem}
\usepackage{mathrsfs}
\usepackage{enumitem}
\usepackage{hyperref}
\usepackage[arrow, matrix, curve]{xy}
\usepackage[nameinlink,capitalize]{cleveref}
\usepackage{caption}
\usepackage{subcaption}

 
\newtheorem{theorem}{Theorem}[section]
\newtheorem{lemma}[theorem]{Lemma}
\newtheorem{corollary}[theorem]{Corollary}
\newtheorem{proposition}[theorem]{Proposition}

\theoremstyle{definition}
\newtheorem{definition}[theorem]{Definition}

\newtheorem{example}[theorem]{Example}

\theoremstyle{remark}
\newtheorem{Remark}[theorem]{\rm \bf Remark}
\numberwithin{equation}{section}

\newcommand{\hgf}[1]{{}_2F_1( #1 )}

\newcommand{\ltfrac}[2]{\mbox{\Large$\frac{#1}{#2}$}}
\newcommand{\mr}[1]{\mathrm{#1}}
\newcommand{\ti}[1]{\widetilde{#1}}

\newcommand{\C}[1]{\mathbb{C}^{#1}}
\newcommand{\bbC}{\mathbb{C}}

\newcommand{\R}[1]{\mathbb{R}^{#1}}
\newcommand{\bbR}{\mathbb{R}}
\newcommand{\N}{\mathbb{N}}
\newcommand{\bz}{\mathbf{z}}
\newcommand{\bw}{\mathbf{w}}

\newcommand{\fg}{\mathfrak{g}}

\newcommand{\fk}{\mathfrak{k}}

\newcommand{\rH}{\mathrm{H}}
\newcommand{\rS}{\mathrm{S}}
\newcommand{\rT}{\mathrm{T}}
 \newcommand{\oPsi}{\widetilde{\Psi}}
\newcommand{\SO}{\mathrm{SO}}

\newcommand{\so}{\mathfrak{so}}
\newcommand{\Hn}{\mathbb{H}^n}
\newcommand{\oHn}{\overline{{\mathbb H}}^n}
\newcommand{\dS}{\mathrm{dS}^n}
\newcommand{\dSC}{\mathrm{dS}^n_{\mathbb{C}}}
\newcommand{\Sn}{\mathbb{S}^n}
\newcommand{\oXi}{\overline{\Xi}}
\newcommand{\Fl}{{}_2F_1}

\let\oldproofname=\proofname
\renewcommand{\proofname}{\rm\bf{\oldproofname}}
\newcommand{\ip}[2]{\langle #1,#2\rangle}

\title[Analytic wavefront set on the De Sitter space]{Spherical Distributions on the De Sitter Space and their Spectral Singularities}

\author{Iswarya Sitiraju}
\address{Department of Mathematics, Louisiana State University, Baton Rouge, LA 70803, USA}
\email{isitir1@lsu.edu}

\begin{document}

\begin{abstract}
    A spherical distribution is an eigendistribution of the Laplace-Beltrami operator with certain invariance on the de Sitter space. Let $G'=O_{1,n}(\R{})$ be the Lorentz group  and $H' = O_{1,n-1}(\R{})$ be its subgroup. The authors \'Olafsson and  Sitiraju have constructed the spherical distributions, which are $H'$-invariant, as boundary values of some sesquiholomorphic kernels. In this survey article we will explore the connections of these kernels with reflection positivity and representations of the group $G = SO_{1,n}(\R{})_e$, which is the connected component of the Lorentz group. We will also discuss the singularities of spherical distributions in terms of their wavefront set.
\end{abstract}
\maketitle
\section{Introduction}
The $n$-dimensional de Sitter space $\dS$, which is a one-sheeted hyperboloid, is a simple model of the universe in special relativity with constant positive curvature. There have been several studies done to understand quantum field theories on the de Sitter space including the papers \cite{BM96, BM04, BV96, BV97} where the authors studied free fields and the related two point functions $\mathcal{W}_{\lambda}(x_1,x_2)$. The theory of interacting quantum fields on the de Sitter space of dimension 2 is discussed in the paper \cite{BJM13}. On the other hand, the authors in \cite{NO22} study some aspects of algebraic quantum field theory on casual symmetric spaces of which the de Sitter spaces is an example. One of the tool used in \cite{NO22} is an open complex domain called the crown domain $\Xi$. 

In \cite{NO18}, the authors showed that the de Sitter space lies on the boundary of $\Xi$ and its conjugate $\oXi$ which are open subsets of the complex unit sphere.  The crown $\Xi$ is holomorphically equivalent to the Lorentzian tuboid $\mathcal{T}^{+}$ in the complexified de Sitter space $\dS_{\C{}}$ defined in \cite{BM96, BM04, BV96, BV97}. Similarily, $\oXi$ is equivalent to the tuboid $\mathcal{T}^-$ in $\dS_{\C{}}$. The de Sitter space lies on the boundary of $\mathcal{T}^{+}$ and $\mathcal{T}^{-}$.

Let $\rho = (n-1)/2$. For $\lambda \in i[0,\infty) \cup (0,\rho)$, the distribution $\mathcal{W}_{\lambda}(x_1,x_2)$ given in \cite{BM96} satisfies the Klein-Gordon equation $\Delta + (\rho^2-\lambda^2)$ in both variables. Moreover, this distribution is the boundary value of some analytic kernels in $\mathcal{T}^+$ and $\mathcal{T}^-$ called ``perikernels". In \cite{NO18, NO20}, the authors have introduced the kernel $\Psi_\lambda$ up to a constant which corresponds to the perikernels and showed that $\Psi_{\lambda}$ was represented as a hypergeometric function. Analogous to $\mathcal{W}_\lambda$, the boundary values of the kernels $\Psi_{\lambda}$ and $\ti{\Psi}_\lambda$ defined below, are studied in \cite{OS23} where

\begin{equation*}\label{eq : pm1}    \Psi_{\lambda}(z,w) =  {}_2F_1\Big(\rho+\lambda,\rho-\lambda; \frac{n}{2}; \ltfrac{1+[z,\Bar{w}]}{2}\Big), \quad  z,w \in \Xi,
\end{equation*}

and
\begin{equation*}\label{eq:pm2}
    \ti{\Psi}_{\lambda}({z},{w}) =  {}_2F_1\Big(\rho+\lambda,\rho-\lambda; \frac{n}{2}; \ltfrac{1+[{z},\Bar{w}]}{2}\Big), \quad  {z},{w} \in \oXi.
\end{equation*}

These kernels were obtained by reflection positivity on sphere. Reflection positivity is one of the Osterwalder-Schrader axioms of constructive quantum field theory. It is a necessary and sufficient condition for a Euclidean field theory that has Euclidean symmetries to analytically continue to a relativistic field theory with Lorentzian symmetries. The initial attempt was done by E. Nelson in the paper \cite{N73} and the breakthrough was done in the paper \cite{OS73, OS75}. 

The kernels $\Psi_\lambda$ and $\ti{\Psi}_\lambda$ are well-defined sesquiholomorphic, positive-definite, G-invariant kernels. The perikernels are holomorphic in the cut domain of $\dS \times \dS$ of the form $\dS \times \dS \setminus \Sigma$ where $\Sigma$ is the set of tuples $(x,y)$ with $[x-y,x-y] \leq 0$. In the paper \cite{OS23}, we showed that the boundary values of the kernels $\Psi_{\lambda}$ and $\ti{\Psi}_\lambda$ are real analytic on the cut domain $\dS \times \dS \setminus \{(x,y): [x-y,x-y] =0 \}$ and have jump discontinuities along the cut. The cut is where the distributions have singularities, we studied these singularities in terms of analytic wavefront sets. 

The wavefront set of a distribution was introduced by L. H\"ormander in 1970. For a distribution $\Theta$ on a real analytic manifold,
the {\it analytic wavefront set} $WF_A(\Theta)$  describes the set of points  where $\Theta$ is not given by a real-analytic function and the direction in which the singularity occurs (see \cite{H90}).  The wavefront set is a crucial concept in quantum field theory(QFT). One of the initial papers using wavefront sets in QFT was \cite{Di79}. Later the wavefront set was brought into the context of Hadarmard distributions in \cite{RM}. It was shown that the Hadamard condition of a two point distribution of a quasi-free quantum field is equivalent to a condition on its wavefront set. In algebraic quantum field theory, the condition on the wavefront set of the states of quantum fields is related to the Reeh-Schlieder property (see \cite{ SVW02, V99}). It was also conceptualized in the context of unitary representations of Lie groups in \cite{HR} and studied for induced representations in \cite{HHO}.

In this article we will collect the main theorems of \cite{OS23} and their connections to representation theory(\cite{D09}) and reflection positivity (\cite{NO20}). The geometrical setup is given in Section 2. The connection with reflection positivity on the sphere is established in Section 3 and Section 4 explains the representation theoretic perspective. In Section 5 and 6, the spherical distributions and their analytic wavefront sets are reviewed.

\subsection*{Notations:} We use the following notations throughout the article:
%he notations we will be using are as follows:\\
%\par
%\textbf{Notations :} 
\begin{itemize}
     \item $\mr{G} = \mr{SO}_{1,n}(\R{})_e$,
    \item $\mr{H} = \mr{SO}_{1,n-1}(\R{})_e\subset G$,
     \item  $\mr{K} = \mr{SO}_{n}(\R{})\subset G$,
     \item  $[z,w] = -z_0w_0 +z_1w_1 +...+ z_nw_n$ for $z,w \in \C{1+n}$,
     \item $\R{1,n} = (\R{1+n},[\, , \,])$,
     \item $\dS =\{x\in \R{1+n}\mid [x,x]=1\} = G/H$,
 \item $ \dSC = \{z\in \C{1+n}\mid [z,z]=1\}$,
  \item $\mathbb{H}^n = \{x\in i\R{1+n}\mid x_0 > 0, [x,x] = 1\} \simeq G/K\subset \dSC$,
  \item $\oHn= \{x\in i\R{1+n}\mid x_0 < 0,  [x,x] = 1\} \simeq G/K \subset \dSC$,
  \item $\mathbb{S}^n = \{(ix_0,\mathbf{x}) \mid x_0^2 + \mathbf{x}^2 = 1\}$,
  \item $\Sn_{\pm} = \{(ix_0,\mathbf{x}) \in \Sn\mid \pm x_0>0\}$,

   \item $\Gamma^{\pm}(x) = \{ y \in \dS \mid \text{for}\, x \in \dS, [y-x,y-x] < 0, \pm y_0 > x_0 \}$,
     \item $\Gamma(x) = \Gamma^{+}(x)\cup \Gamma^{-}(x)$,
     \item $\rho = (n-1)/2$ for $n \geq 2$.
\end{itemize} 
%-------------------------------------------------------------------------------------------------------------
\section{Geometrical setup} \label{sec:ds}

We are going to follow the geometrical setup described in the paper \cite{OS23}. However, the setups in \cite{NO20, BM96, OS23} are all equivalent.

Let $z,w \in \C{1+n}$ and let the bilinear form on $\C{1+n}$ be given by
\[ [z,w] = -z_0w_0 + \sum_{j=1}^n z_j w_j = -z_0w_0 + \bz\cdot \bw . \]
We say that a vector $v \in \R{1,n}$ is time-like if $[v,v]<0$ and space-like if $[v,v]>0$.

\subsection{The hyperboloid and the de Sitter space}
    The {\it hyperbolic spaces} $\Hn$ and $\oHn$  are given as follows:
\[\Hn =\{x\in i\R{1+n}\mid [x,x]=1, x_0>0\}\quad\text{and}\quad \oHn =\{x\in i\R{1+n}\mid [x,x]=1, x_0<0\}.\]
The de Sitter space $\dS$ is described as
\[  \dS = \{x \in \R{1+n}\mid [x,x] = 1 \}.\]
All the spaces defined above are closed submanifolds of the complex manifold 
\[\dS_\bbC =\{z\in \C{1+n}\mid [z,z]=1\}.\]

Let $x \in \dS$, then we denote the future (past) cone of $x$ as $\Gamma^{+}(x) (\Gamma^-(x))$ where
\[\Gamma^{\pm}(x) := \{y\in \dS\mid [y-x,y-x] <0, \pm (y-x)_0 >0\}.\]

For $x\in \dS$ the set $\{y \in \dS\mid [y-x,y-x]=0\}$ is called the light cone of $x$ in $\dS$.

%Let $\tau : G \rightarrow G$ be the involution given by $\tau(g) = JgJ$, where $J$ is the orthogonal reflection in the hyperplane $x_n = 0$. Furthermore,
%$$\mathfrak{g} = \mathfrak{h} \oplus \mathfrak{q}$$
%with $\mathfrak{h} = \mathrm{ker}(\tau -1)$ and $\mathfrak{q} = \mathrm{ker}(\tau + 1)$.
%Then we have that 
%\[\mr{T}_{e_n} \dS \cong \mathfrak{q} \cong \R{1,n-1}.\]

Let $G=\SO(1,n)_e$ be the connected component of identity of the isometry group of $[\cdot,\cdot]$. 
We denote by $K=\SO (n)$ the maximal compact subgroup
\begin{align*}
    K& =\{k\in G\mid g\cdot e_0= e_0\}, \\[2mm]
    A& = \Bigg \{a_t = \begin{pmatrix}
        \cosh t & 0 & \sinh t\\
        0& I_{n-1}&0\\
        \sinh t & 0 & \cosh t
    \end{pmatrix} : t \in \R{}\Bigg \},
\end{align*}
and by
\begin{align*}
    H &= \{h\in G\mid h\cdot e_n = e_n\}= \SO (1,n-1)_e .
\end{align*}

The group $G$ acts transitively on $\Hn$ and $\dS$ and,
\[\Hn = G\cdot ie_0 \simeq G/K\simeq \oHn = G\cdot (-ie_0) \quad\text{and}
\quad \dS = G\cdot e_n \simeq G/H .\]
 
The de Sitter space is a homogenous Lorentzian manifold with local co-ordinates at each point is given from its tangent space. For $x\in \dS$ we have
\[\rT_x(\dS ) =\{y\in \R{1+n}\mid [x,y]=0\}\cong \R{1,n-1} . \]

The following lemma has been proved in \cite[Lemma 6.3]{NO20}.

\begin{lemma} \label{lemma:KAH}
    $G = HAK = KAH$ and,
   $$G/H =  KA.e_n = \dS.$$
\end{lemma}

There exists a unique up to a constant $G$-invariant measure on $\dS$. For more discussions see \cite[p.159]{D09}. Let $\Delta$ be the Laplace-Beltrami operator obtained from the Lorentz metric on $\dS$ induced from its ambient space $(\R{1+n},[\cdot,\cdot])$. It is known from \cite{F79} that the algebra of polynomials in $\Delta$ is the algebra of $G$-invariant differential operators on $\dS$.

\subsection{The Crown $\Xi$, $\oXi$}
The complex crown  $\Xi$ of a non compact Riemannian symmetric space $G/K$ has the property that the eigenfunctions of the algebra of invariant differential
operators extends to $\Xi$.
It was introduced in \cite{AG90}. It was studied by several authors including the articles \cite{GK02a,GK02b,KSt04}  and \cite{GK02b} showed that the non-compactly causal symmetric spaces \cite{HO97}, including the de Sitter space, can be realized
as open orbits in the boundary of the crowns. We will now discuss how the kernel $\Psi_\lambda$ obtained from reflection positivity is connected to representation theory.

Let $h \in \so (1,n)$ be the operator 
\[h(x_0,x_1, \ldots , x_{n-1}, x_n) = (x_n,0, \ldots ,0,x_1).\]

The unit sphere is realized in $i\R{}e_0 + \R{n}$ by
 $\Sn =\{x\in i\R{}e_0 + \R{n}\mid [x,x]=1\}$. Set
 $\Sn_+ = \{ x\in \Sn\mid x_0>0\}$ and $\Sn_- = \{ x\in \Sn\mid x_0 < 0\}.$

The crown domain of $\Hn$ is described as 
\[\Xi =  G\cdot\Sn_+ = G\exp (i(-{\pi}{2},{\pi}{2}) h )\cdot (ie_0) \]
and for $\oHn$ as
\[\oXi = G\cdot\Sn_- = G\exp i(-\pi/2, \pi /2)h \cdot (-ie_0) .\]

 The crown domain $\Xi  \subset \SO_{d+1}(\C{})(ie_0) \simeq G_\C{}/K_\C{} \simeq \dS_{\C{}}$. Similarly, for $\oXi$.

\begin{Remark}The crown domain depends on $(\fg,\fk)$ where $\fg,\fk$ are the Lie algebra of $G$ and $K$, respectively. It does not depend the choice of Lie group G with Lie algebra $\fg$.
\end{Remark}

The crown domains $\Xi (\oXi)$ contains $\Sn_+(\Sn_{-})$ and $\Hn(\oHn)$ as submanifolds. More details on the crown domains can be found in \cite{NO20, OS23}. 

The following statement that the de Sitter space lies on the boundary of $\Xi$ and $\oXi$ was proven in \cite{NO20, OS23}. Hence, the crown domains form a bridge between the Riemannian manifolds $\Hn, \oHn$ and the Lorentzian manifold $\dS$.

\begin{lemma} $\dS=\partial \, \Xi \cap \partial\, \oXi$.
\end{lemma}

\begin{Remark}
    It was shown in \cite{OS23, BM96} that around each point $x \in \dS$, the crown can be represented locally as a tuboid of the form $U + i\Omega'$ where $U$ is an open set  and $\Omega'$ is a pointed cone in the tangent space of $x$.
\end{Remark}

%----------------------------------------------------------------------------------------------------

\section{Reflection Positivity, Kernels $\Psi_\lambda$ and $\ti{\Psi}_\lambda$}

We will now recall reflection positivity on the sphere \cite{NO20}(see also \cite{NO22}), which leads to the
positive definite kernel $\Psi_\lambda$ (with a different
normalization in \cite{NO20}) given for $\lambda \in i[0,\infty) \cup [0,\frac{n-1}{2})$ and $\rho = (n-1)/2$
by
\begin{equation}\label{def:PsiLambda}
\Psi_\lambda (z,w) = {}_2F_1\left( \rho + \lambda, \rho-\lambda ; \frac{n}{2}; 
\frac{1+[z,\overline{w}]}{2}\right), \quad z,w \in \Xi.
\end{equation}

We will also consider the following kernel
\begin{equation}\label{def:PsiLambda}
\ti{\Psi}_\lambda (z,w) = {}_2F_1\left( \rho + \lambda, \rho-\lambda ; \frac{n}{2}; 
\frac{1+[z,\overline{w}]}{2}\right), \quad z,w \in \oXi.
\end{equation}

As both $n$ and $\lambda$ are fixed most of the time, we simplify our notation and write
\[\Fl (z)  = {}_2F_1\left( \rho + \lambda,  \rho-\lambda ; \frac{n}{2}; z\right).\]
Here  ${}_2F_1(a,b;c;z)$ denotes the Gauss hypergeometric function
\[{}_2F_1(a,b;c;z) = \sum_{n=0}^\infty \frac{(a)_n(b)_n}{(c)_n}\frac{z^n}{n!},\]
where $(d)_n = d(d+1)\cdots (d+n-1)$,  $c\not\in -\N_0$ and $|z|<1$. The hypergeometric function ${}_2F_1$ extends to a holomorphic function
on $\bbC\setminus [1,\infty)$ (see \cite{LS}). 

Consider the unit sphere $\Sn$ realized in  $i\R{}e_0 + \R{n}$. Let $\sigma: \Sn \rightarrow \Sn$ be the reflection $\sigma(x_0, {\bf x}) = (-x_0, {\bf x})$.
Let us denote by $\square$ the Laplacian on $\Sn$. We will now consider a distribution $\Phi_\lambda$ on $\Sn \times \Sn$ as follows:
\[\Phi_\lambda(\phi \otimes \overline{\psi}) = \int_{\Sn} \overline{\phi(x)}(-\square + (\rho^2-\lambda^2))^{-1}\psi(x)d\mu(x) \quad \text{for} \quad \phi, \psi \in C_c^{\infty}(\Sn)\]

where $(-\square + (\rho^2-\lambda^2))^{-1}$ is a bounded positive operator on $L^2(\Sn)$ for $\lambda \in i[0, \infty)\cup(0,\rho)$.

It was proven in \cite[Sec 2]{NO20} that the distribution $\Phi_{\lambda}$ is {\it reflection positive } with respect to $(\Sn, \Sn_+, \sigma)$. That is, the distribution $\Phi_\lambda^{\sigma} = \Phi_\lambda \circ (id, \sigma)$ is positive definite on $\Sn_+ \times \Sn_+$. It was also proven that the distribution $\Phi_\lambda^{\sigma}$ is given by the kernel $\Psi_\lambda(x,y)$ for $x,y \in \Sn_+$. For a general Riemannian manifold a similar construction is done to get a reflection positive distribution (see \cite{Di04, JR08}).

This kernel has a natural extension on $\Xi \times \Xi$. Part (1) of the following theorem is in \cite{NO20,NO22}, part (2) can be found in \cite{OS23} and part (3) follows by part (2).
\begin{theorem} Let $\rho = \frac{n-1}{2}$. Then
\begin{itemize}
    \item[\rm (1)] The kernel $\Psi_{\lambda}(z,w)$  is a positive definite 
    $G$ invariant kernel  on $\Xi \times \Xi$   which is holomorphic in first 
variable and anti-holomorphic in the second variable. It is given by 
    $$\Psi_{\lambda}(z,w) = \,
     {}_2F_1\left(\rho+\lambda,\rho-\lambda; \frac{n}{2}; \ltfrac{1+[z, \overline{w}]}{2}\right), \quad  z,w \in \Xi.$$
\item[\rm (2)] Let $z,w\in\Xi$. Then
\[\overline{\Psi_\lambda (z,w)} = \Psi_\lambda( \overline{z},\overline{w}). \]
    \item[\rm (3)]  The kernel $\widetilde{\Psi}_{\lambda}(z,w)$ is
a $G$-invariant positive definite kernel on
 $\oXi \times \oXi$ holomorphic in the first variable and
 anti-holomorphic in the second variable given by:
    $$\ti{\Psi}_{\lambda}(z,w) = \, {}_2F_1\Big(\rho+\lambda,\rho-\lambda; \frac{n}{2}; \ltfrac{1+[{z}, \sigma(w)]}{2}\Big), \quad  z,w \in \oXi.$$
\end{itemize}
   We also have that ${\Psi}_{\lambda}(z,w) = \overline{\ti{\Psi}_{\lambda}(\bar{z},\bar{w})} = \ti{\Psi}_\lambda(\Bar{w},\Bar{z})$. 
\end{theorem}

The following lemma has been proved in \cite[Lem. 6.4]{NO20}.
\begin{lemma}\label{l1} We have
\[[\dS ,\Xi]\cap \bbR = [\dS , \oXi]
\cap \bbR = (-1,1).\] 
\end{lemma}
%\begin{proof} We prove this for $\Xi$. The claim for $\oXi$ then follows by
%$\oXi = \sigma (\Xi )$ and $\sigma (\dS) = \dS$. As $\dS$ is $G$ invariant is is enough
%by Lemma \ref{lemma:crown} to show that $[\dS \cap \rS^n_+]\cap \bbR = (-1,1)$. For
%that let $x = (x_0,\bx)\in \dS $ and $y= (iy_0,\by)\in \rS^n_+$. Then $[x,y]\in \bbR$ if
%and only if $x_0y_0=0$. As $y_0>0$ this happens if and only if $x_0$ which is equivalent to
%$\bx^2 =1$. As $\by^2 < 1$ it follows that $\bx \by \in (-1,1)$. If $s\in (-1,1)$ take
%$y = \sqrt{1-t^2}e_0 + te_n\in \rS^n$ and $x=e_n$. Then $[x,y] =t$.
%\end{proof}

From this and the properties of the hypergeometric function we get:

\begin{proposition}\label{cont}
 The kernel $\Psi_{\lambda}$ can be extended continuously to $\Xi \times (\dS\cup\Xi) $ and the kernel $
 \oPsi_{\lambda}$ can be extended continuously to $\oXi \times( \dS\cup\oXi)$.
\end{proposition}
For $y \in \dS$ and $z = e_n$, we have that $(1 + [z,y])/2 \notin [1, \infty)$ iff $y_n < 1$. In particular for $y \in\dS$ we have that
$y \mapsto \Psi_\lambda (y,e_n), \Psi_\lambda (e_n,y), \oPsi_\lambda (y,e_n),\oPsi (e_n,y)$ are analytic on $\{y \in \dS \mid y_n< 1\}$. We will discuss these singularities in \cref{sec:bv} and \cref{sec:ws}.

%-----------------------------------------------------------------------------------------------------------------

\section{Representation Theory Perspective}\label{sec: rep} 

We will now discuss how the kernel $\Psi_\lambda$ obtained by reflection positivity is related to representations.
We have the Iwasawa decomposition $G = KAN$. We write $g = k(g) a(g) n(g)$ and $G$ acts on $\rS^{n-1}$ by 
$g\cdot v= k(g)v$.  The principal series representation
 $\pi_\lambda$  with spectral
parameter $\lambda$ acting on the Hilbert space $\rH_\lambda=L^2(\rS^{n-1})$ is given by
\[\pi_{\lambda }(g) f(v) = a(g^{-1}k)^{-\lambda - \rho} f(g^{-1}\cdot v)\]
where  
$v \in\rS^{n-1}$,  $g\in G$  and  $f\in L^2(\rS^{n-1})$.
The constant function $e_\lambda (v) = 1$ is $K$-invariant with norm $1$ and
the associated spherical function is
\[\phi_\lambda (g) = \ip{\pi_\lambda (g)e_\lambda}{e_\lambda} =
\int_{\rS^n} a(g^{-1}v)^{-\lambda - \rho}dv .\] 
We note that $g\mapsto \pi_\lambda (g)e_\lambda$ is right $K$-invariant, hence
$\pi_\lambda (z)e_\lambda$ is well defined for $z$ and $w$ in $\Hn$ and  the kernel $\Psi_\lambda$ is given by
\[\Psi_\lambda (z,w) = \ip{\pi_\lambda (z)e_\lambda}{\pi_\lambda (w)e_\lambda}.\]
Therefore,
\[\phi_\lambda (x) =
\Psi_\lambda (x,ie_0)=  {}_2F_1 \left( \rho + \lambda, \rho -\lambda ; \frac{n}{2}; 
\frac{1+ix_0}{2}\right),\quad ix\in\Hn,\]
is an eigenfunction of the algebra of $G$-invariant differential operators on $\Hn$ (see \cite{D09}).

%For $g \in G$ we have $g\mr{exp}(-ith)e_0 \in \Xi$ and, the orbit map
%\[(-\pi/2,\pi/2)\mapsto \pi_\lambda (g\exp (-ith ))e_\lambda 
%\]
%is analytic and 
%\begin{equation}\label{eq:eH}
%e_\lambda^H = \lim_{t\to \pi/2} \pi_\lambda (\exp (-ith ))e_\lambda
%\end{equation}
%exists in the space of distribution vectors $\rH_\lambda^{-\infty}$ and defines an $H$-invariant distribution
%vector \cite[Sec. 5]{FNO23}. Furthermore $\pi^{-\infty}_\lambda (\varphi)e_\lambda^H $ belongs to the space of smooth vectors $\rH^\infty_\lambda$ for
%$\varphi \in C_c^\infty (G/H)$, see \cite[Chap. 7]{NO18}
% Hence
%\[\Theta_\lambda (\varphi ) = \ip{e_\lambda^H}{\pi_\lambda^{-\infty}(\varphi)e_\lambda^H}\]
%defines a $H$-invariant distribution. Furthermore,
%\begin{equation}\label{eq:spherical}
%   \Delta \Theta_\lambda = ( \rho^2 - \lambda^2)\Theta_\lambda . 
%\end{equation}

We have that $\exp (-ith)e_n$ is an element in $\Xi$. It was proven in \cite{GKO04}, \cite[Thm. 2.1]{vdBD88} that for $\phi \in C_c^{\infty}(\dS)$ the limit
\begin{equation}\label{eq:distlim}
    \Theta^\lambda (\varphi) = \lim_{t \to \pi/2^-}  \int_{\dS} \overline{\varphi (y)} 
\Psi_{\lambda}(\exp (-ith)e_n,y)d\mu_{\dS}(y),
\end{equation}

defines an $H$-invariant distribution (see also \cite[Thm.1]{BD92}) using representation theoretic techniques. Another proof using Hardy space approximation was given in \cite{NO18} and a different proof was given in \cite{FNO23}. A third proof independent of representation theory is dealt in \cite{OS23} which we will see in the next section. 

Furthermore, we also obtain that
\begin{equation}\label{eq:spherical}
  \Delta \Theta^\lambda = ( \rho^2 - \lambda^2)\Theta^\lambda . 
\end{equation}

Similar discussion holds for $\oXi$, $\oHn$ and $\ti{\Psi}_\lambda$.
 
%------------------------------------------------------------------------------------------------

%------------------------------------------------------------------------------------------------------------------------------------
\section{ Spherical Distributions as Boundary values of holomorphic functions}\label{sec:bv}

We use the usual notation 
$\mathcal{D}(\dS) = C_c^{\infty}(\dS)$ with the
standard topology. We consider the distributions as anti-linear functionals on $\mathcal{D}(\dS)$.  From now onwards we will denote the elements in $\oXi$ as $\bar{z}$, since $\overline{\bar{z}}=z$ lies in $\Xi$.

Let $\Theta$ be a distribution on $\dS$. Then $G$ acts on $\Theta$ by 
\[\pi_{-\infty}(g)\Theta (\varphi) = \Theta(\pi_{\infty}(g^{-1})\varphi), \quad \varphi \in \mathcal{D}(\dS),\]
where $\pi_{\infty}(g)\varphi(x) = \varphi(g^{-1}\cdot x)$.
\begin{definition}
Let $H$ be a closed subgroup of $G$. We say that a distribution $\Theta$ is $H$-invariant if $\pi_{-\infty}(h)\Theta = \Theta$ for all $h \in H = G_{e_n}$.
\end{definition}
\begin{definition}
    A distribution $\Theta$ is said to be a spherical distribution if it is an $H$-invariant  eigendistribution of the Laplace-Beltrami operator $\Delta$.
\end{definition}

Let  $G' = O_{1,n}(\R{})$ be the full Lorentz group and $H' = O_{1,n-1}(\R{})$ be its subgroup which fixes $e_n$. Then $G'/H'$ is isomorphic to $\dS$. It was proven in \cite{D09}[Thm.9.2.5] that the dimension of $H'-$invariant spherical distributions with eigenvalue $\rho^2-\lambda^2$ is two. The analytic wavefront sets and the characterization of an $H'$-invariant spherical distribution based on its singularities have been studied in \cite{OS23}. However, in the case of $(G,H)$, it is a hypothesis that the dimension of $H$-invariant spherical distribution is four based on \cite{OSe80}. Two of those are distributions that are constructed as boundary values of analytic functions in this section. Others will be dealt in another paper.

From \cref{eq:spherical}, we observe that the distribution $\Theta^\lambda$ is a spherical distribution with eigenvalue $(\rho^2-\lambda^2)$.
Let $\mathcal{D}_{\lambda}^{H}(\dS)$ be the space of spherical distributions on the de Sitter space with $\Delta (\Theta) = (\rho^2-\lambda^2) \Theta$.   

Before understanding the limit, we will need the following theorem \cite{OS23}.

\begin{theorem} \label{thm:hgf}
The limit $\lim_{y \rightarrow 0} \; {}_2F_1(\rho + \lambda, \rho - \lambda, n/2, x \pm iy)$ for $y >0$ exists in the sense of distributions where for $\rm{Re}(z) > 1$ the limit converges uniformly on compact sets. We have that for $1<x<2$, if $n$ is odd 
\begin{small}

\begin{equation}\label{eq:2}
    \begin{split}
      &\hgf{x\pm i0} =  \frac{\Gamma(n/2)\Gamma((2-n)/2)}{\Gamma(1/2 + \lambda)\Gamma(1/2-\lambda)} {}_2F_1\left(\rho + \lambda,\rho-\lambda; \frac{n}{2};1-x\right) \\
   &+ e^{\mp i\pi(\frac{2-n}{2})}(x-1)^{\frac{2-n}{2}}  \ltfrac{\Gamma(n/2)\Gamma((n-2)/2)}{\Gamma(\rho + \lambda)\Gamma(\rho-\lambda)} {}_2F_1\left(1/2- \lambda,1/2+\lambda; \frac{4-n}{2};1-x\right)
    \end{split}
  \end{equation} 
  \end{small}
and if $n$ is even $\hgf{x \pm i0} =$
\begin{small}
\begin{equation}\label{eq:3}
\begin{split}
     &  \ltfrac{\Gamma(n/2)}{\Gamma(\rho + \lambda)\Gamma(\rho-\lambda)}\sum_{k=0}^{n/2-2}\ltfrac{(-1)^{k}(n/2-k-2)!(1/2+\lambda)_k(1/2-\lambda)_k}{k!}(1-x)^{k+1-\frac{n}{2} }\\
    &+ \ltfrac{(-1)^{\frac{n-2}{2}}\Gamma(n/2)}{\Gamma(1/2 + \lambda)\Gamma(1/2-\lambda)}\sum_{k=0}^{\infty}\ltfrac{(\rho+\lambda)_k(\rho-\lambda)_k}{k!(n/2 - 1 + k)!}\Bigg[\psi(k+1) + \psi(n/2 + k) \\
   & - \psi(\rho+\lambda +k) - \psi(\rho-\lambda +k) - \ln (x-1) \pm i\pi\Bigg](1-x)^{k} ,
\end{split}
\end{equation}
\end{small}
where $\psi(z) = \Gamma'(z)/\Gamma(z)$.
    Furthermore, the behaviour of the hypergeometric function near $z=1$ as distributions is given as follows: for $n = 2$, 
    \begin{equation}
        {}_2F_1(z) \approx  \ltfrac{1}{\Gamma(\rho+\lambda)\Gamma(\rho-\lambda)} (-\ln{(1-z}))
    \end{equation}
    and for $n \geq 3$,
    \begin{equation}
        \hgf{z} \approx \frac{\Gamma(n/2)\Gamma((n-2)/2)}{\Gamma(\rho+\lambda)\Gamma(\rho-\lambda)}(1-z)^{\frac{2-n}{2}}. \label{4}
    \end{equation}
\end{theorem}
For $x > 2$, ${}_2F_1(x \pm i0)$ can be analytically continued.

In \cite{OS23} we have proven that
\begin{theorem}\label{thm:psi}
 For $n \geq 2$ and $\lambda \in i[0,\infty) \cup (0,\rho)$,
\begin{enumerate}
    \item the limits $\lim_{t \rightarrow \pi/2}\Psi_\lambda({\rm exp}(-ith)e_n,y)$ in $\Xi$ and $\lim_{t\rightarrow \pi/2}\ti{\Psi}_\lambda({\rm exp}(ith)e_n,y)$ in $\oXi$ taken in the sense of \cref{eq:distlim} converge to  distributions $\Psi^{\lambda}$ and $\ti{\Psi}^{\lambda}$ respectively,  on $\dS$.\\
    \item   The limits also satisfies the equation $(\Delta - m^2)\Psi^{\lambda} = 0 = (\Delta - m^2)\ti{\Psi}^{\lambda}$. Hence, are spherical distributions.\\
    \item Moreover, $\Psi^{\lambda}$ and $\ti{\Psi}^{\lambda}$ can be represented as analytic functions in the following regions:
    \begin{align*}
        \Psi^{\lambda}(y) = \begin{cases}
       {}_2F_1\Big(\frac{1+y_n}{2}\Big) &\text{if $y \notin \overline{\Gamma(e_n)}$},\\
       {}_2F_1\Big(\frac{1+y_n}{2} - i0\Big)  &\text{if $y \in \Gamma^+(e_n)$}, \\
    {}_2F_1\Big(\frac{1+y_n}{2} + i0\Big)  &\text{if $y \in \Gamma^-(e_n)$};
          \end{cases}\\[2mm]
        \ti{\Psi}^{\lambda}(y) = \begin{cases}
       {}_2F_1\Big(\frac{1+y_n}{2}\Big) &\text{if $y \notin \overline{\Gamma(e_n)}$},\\
       {}_2F_1\Big(\frac{1+y_n}{2} + i0\Big)  &\text{if $y \in \Gamma^+(e_n)$}, \\
    {}_2F_1\Big(\frac{1+y_n}{2} - i0\Big)  &\text{if $y \in \Gamma^-(e_n)$}.
          \end{cases}  
    \end{align*} 
\end{enumerate}

\end{theorem}

%-----------------------------------------------------------------------

\section{Wavefront set of spherical distributions on $\dS$} \label{sec:ws}

We will now briefly introduce the wavefront sets and study these for the distributions $\Psi_\lambda$ and $\ti{\Psi}_\lambda$.

Let $X \subset \R{1,n}$ be an open subset.
Suppose $\Theta$ is a distribution with compact support then we can define the Fourier transform of $\Theta$ at $\xi\in (\R{1,n}\setminus 0)$ as follows:
$$\widehat{\Theta}(\xi) = \Theta(e^{2\pi i [x,\xi]}),$$
where $[x,\xi] = -x_0\xi_0 + x_1\xi_1+...+x_{n}\xi_{n}$.

 The definition of analytic wavefront set follows from \cite[Proposition 8.4.2]{H90}. 

\begin{definition}
    If $X$ is an open subset of $\R{1,n}$ and $\Theta \in \mathcal{D}'(X)$, we define $WF_A(\Theta)$ to be the complement in $X \times (\R{1,n}\setminus 0)$ of the set $(x_0,\xi_0)$ such that there is an open neighbourhood $U \subset X$ of $x_0$, a conic neighbourhood $\Gamma$ of $\xi_0$ and a bounded sequence of $\Theta_N \in \mathcal{E}'(X)$ which is equal to $\Theta$ in $U$ and satisfies 
$$|\widehat{\Theta_N}(\xi)| \leq C^{N+1}(N/|\xi|)^N \quad N=1,2,...$$
when $\xi \in \Gamma$ and for some $C>0$.
\end{definition}

\begin{Remark}
    The $WF_A(\Theta)$ is a closed conic set, that is if $(x,\xi) \in WF(\Theta)$, then for $\tau > 0$, $(x,\tau \xi) \in WF(\Theta)$.
\end{Remark}

The next theorem describes the analytic wavefront sets of distributions which are boundary values of analytic functions. Let $\Gamma$ be an open convex cone, then the {\it dual cone} $\Gamma^{\circ}$ is defined as 
$$\Gamma^{\circ} = \{ \eta \in \R{1+n}\mid \eta_0\xi_0 + ...+ \eta_n\xi_n \geq 0, \;\forall \xi \in \Gamma\}.$$
\begin{theorem}\label{bd}
Let $X \subset \R{1,n}$ be an open set and $\Gamma$ an open convex cone in $\R{1,n}$ and for some $\gamma > 0 $,
$$Z = \{ z \in \C{1+n} \mid \rm{Re}\, z \in X, \rm{Im}\, z \in \Gamma, |\rm{Im} z| < \gamma\}.$$
If $\Theta$ is an analytic function in $Z$ such that 
$$|\Theta(z)| \leq C |\rm{Im} \,z|^{-N}$$
for some $N \in \mathbb{N}$ and some constant $C >0$, the  $\underset{y\searrow 0}{\rm{lim}}\Theta(.+iy) = \Theta_0$ exists in terms of distribution and is of order $N$.  We also have that 
$$ WF_A (\Theta_0) \subset X \times (\Gamma^{\circ} \setminus 0).$$
\end{theorem}
\begin{proof}
    The proof follows from Theorem 3.1.15 and Theorem 8.4.8 in \cite{H90}. 
\end{proof}

\begin{example}\label{eg:hgf}
    Let $\Theta = {}_2F_1(x + i0)$, the boundary value of the holomorphic function ${}_2F_1(x +iy)$ for $y>0$. From \cref{thm:hgf} we have that it is a distribution which has an analytic singularity at $x=1$. As a result of \cref{bd}, the analytic wavefront set is $$WF_A({}_2F_1(x + i0)) = \{(1, \tau)\mid \tau > 0\},$$ 
    and it also follows that 
    \begin{equation*}\pushQED{\qed}
       WF_A({}_2F_1(x - i0))= \{(1, \tau)\mid \tau < 0\}.\qedhere  
    \end{equation*}
  \end{example}

The following is the second main theorem of the paper \cite{OS23}. 

\begin{theorem}\label{thm:wf}
    The analytic wavefront sets of the spherical distributions $\Psi^\lambda$ and $\ti{\Psi}^\lambda$ are given by 
  \begin{align*}
      WF_A({\Psi}^{\lambda}) &= 
  \{(e_n,v)\mid v_0 < 0 \} \cup \{(e_n +v, \tau (-v_0, v_1,...,v_{n-1}))\mid v_0>0\} \\ 
  &\cup\{(e_n+v,\tau (v_0, -v_1,...,-v_{n-1}))\mid v_0 <0\} ; \\
   WF_A(\ti{\Psi}^{\lambda}) &= 
  \{(e_n,v)\mid v_0 > 0 \}\cup \{(e_n +v, \tau (v_0, -v_1,...,-v_{n-1}))\mid v_0>0\} \\ 
  &\cup \{(e_n+v,\tau (-v_0, v_1,...,v_{n-1}))\mid v_0 <0\} ,
  \end{align*}
  for $\tau > 0$ and $v \in T_{e_n}(\dS)$ such that $[v-e_n,v-e_n] = 0$. 
  
  %Moreover,  the following holds for a non-zero spherical distribution $\Theta$ on $\dS$: 
   %\begin{enumerate}
    %   \item $WF_A(\Theta)\subset WF_A(\Psi^{\lambda}) \cup WF_A(\ti{\Psi}^{\lambda})$.
     %   \item If  $WF_A(\Theta) = WF_A(\Psi^{\lambda}) $ then there exists a nonzero constant $c$ such that
      % $\Theta=c\Psi^{\lambda}$.
     %\item If  $WF_A(\Theta)=WF_A(\ti{\Psi}^{\lambda})$ then  there exists a nonzero constant $c$ such that
     %   $\Theta=c\ti{\Psi}^{\lambda}$. 
   % \end{enumerate} 
\end{theorem}

In \cref{fig:wavefront set} we can see the analytic wavefront set in $T_{e_n}\dS$. The proof involves calculating the analytic wavefront sets in local co-ordinates using \cref{bd} and propagation of singularities from \cite{H71}.

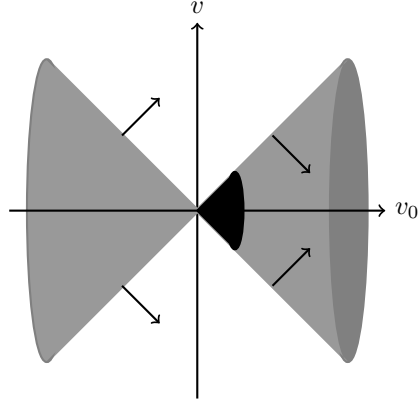
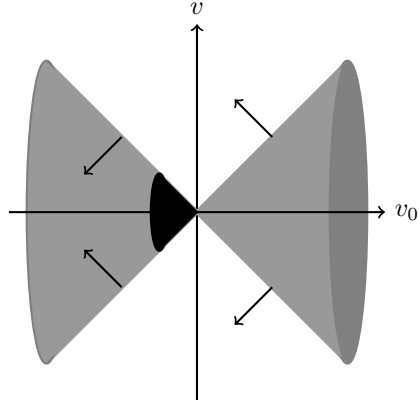
\begin{figure}[!h]
\centering
\begin{subfigure}{0.5\textwidth}
\begin{tikzpicture}[scale = 0.5]
\draw[->,thick] (0,-5)--(0,5) node[above]{$v$};
 \draw[black!40, ultra thick] (-4,-4)--(4,4);
 \draw[black!40, ultra thick] (-4,4)--(4,-4);
 \filldraw[color=black!50, fill = black!50, fill opacity = 0.2,ultra thick] (4,0) ellipse (0.5 and 4) ;
  \filldraw[color=black!50, fill = black!50, fill opacity = 0.2,ultra thick] (-4,0) ellipse (0.5 and 4) ;
 \fill[ color=black!40, opacity=0.2] (4,4)--(0,0)--(4,-4) arc (270: 90:0.5 and 4);
\fill[ color=black!40, opacity=0.2] (-4,4)--(0,0)--(-4,-4) arc (270:90:0.5 and 4);
\fill[color=black!100, opacity=0.2] (1,1)--(0,0)--(1,-1) arc (270:90:0.2 and 1);
 \filldraw[color=black!100, fill = black!150,fill opacity = 0.2,ultra thick] (1,0) ellipse (0.2 and 1) ;
 \draw[->, thick] (-5,0)--(5,0) node[right]{$v_0$};
 \draw[->, black!100, thick] (2,2)--(3,1);
 \draw[->,black!100,thick] (2,-2)--(3,-1);
 \draw[->,black!100,thick] (-2,2)--(-1,3);
 \draw[->,black!100,thick] (-2,-2)--(-1,-3);
 \draw[black!100,thick](0,0)--(1,1);
 \draw[black!100,thick](0,0)--(1,-1);
\end{tikzpicture}
\caption{Analytic wavefront set of $\ti{\Psi}^{\lambda}$.}
\end{subfigure}

\par\bigskip 
\begin{subfigure}{0.5\textwidth}
    \centering
    \begin{tikzpicture}[scale = 0.5]
\draw[->,thick] (0,-5)--(0,5) node[above]{$v$};
 \draw[black!40, ultra thick] (-4,-4)--(4,4);
 \draw[black!40, ultra thick] (-4,4)--(4,-4);
 %\filldraw[color=red!60, fill=red!5, very thick](-1,0) circle (1.5)
 %(2.5,0) ellipse (1.5 and 0.5);
 \filldraw[color=black!50, fill = black!50, fill opacity = 0.2,ultra thick] (4,0) ellipse (0.5 and 4) ;
  \filldraw[color=black!50, fill = black!50, fill opacity = 0.2,ultra thick] (-4,0) ellipse (0.5 and 4) ;
\fill[ color=black!40, opacity=0.2] (4,4)--(0,0)--(4,-4) arc (270: 90:0.5 and 4);
\fill[ color=black!40, opacity=0.2] (-4,4)--(0,0)--(-4,-4) arc (270:90:0.5 and 4);
\fill[color=black!150, opacity=0.2] (-1,1)--(0,0)--(-1,-1) arc (270:90:0.2 and 1);
 \filldraw[color=black!150, fill = black!150,fill opacity = 0.2,ultra thick] (-1,0) ellipse (0.2 and 1) ;
 \draw[->, thick] (-5,0)--(5,0) node[right]{$v_0$};
 \draw[->, black!100, thick] (2,2)--(1,3);
 \draw[->,black!100,thick] (2,-2)--(1,-3);
 \draw[->,black!100,thick] (-2,2)--(-3,1);
 \draw[->,black!100,thick] (-2,-2)--(-3,-1);
 \draw[black!100,thick](0,0)--(-1,1);
  \draw[black!100,thick](0,0)--(-1,-1);
\end{tikzpicture}
\caption{Analytic wavefront set of ${\Psi}^{\lambda}$.}
\end{subfigure}
\caption{In this figure, the tangent space at a point $e_n$ has been identified with its cotangent space at $e_n$. The light grey region, which is the light cone of 0, is the analytic singular support. The arrows and the dark grey cone are the directions in which the singularity occurs at that point.}
\label{fig:wavefront set}
\end{figure}

Using the analytic wavefront sets we can prove that the distributions can not vanish on any non-empty open set $O$ of $\dS$.

The following theorem is due to Strohmaier, Verch and Wollenberg, see \cite[Proposition 5.3]{SVW02}.

\begin{theorem}\label{thm:supp}
    Let X be a real analytic manifold and $\Theta \in \mathcal{D}'(X)$. If $WF_A(\Theta) \cap -WF_A(\Theta) = \emptyset$ then for an open region $O$ in $X$
    $$\Theta|_O= 0 \Rightarrow \Theta=0 ,$$
    where $-WF_A(\Theta) = \{(x,\xi)\mid (x, -\xi) \in WF_A(\Theta)\}$.
\end{theorem}

Since the wavefront sets of the distributions $\Psi^{\lambda}$ is such that $WF_A(\Psi^{\lambda}) \cap -WF_A(\Psi^{\lambda}) = \emptyset$,   which is same for $\ti{\Psi}^{\lambda}$, immediately we obtain:

\begin{corollary}
The distributions $\Psi^{\lambda}$ and $\ti{\Psi}^{\lambda}$ can not vanish on any open regions of $\dS$.
\end{corollary}

\begin{Remark}
    Some of the questions that we are going to address in the future are as follows: 
    \begin{itemize}
     \item Can we construct all four spherical distributions as boundary values of some analytic kernels for the pair $(SO_{1,n}(\R{})_e,$ $SO_{1,n-1}(\R{})_e)$ ?
       \item Can we obtain similar results of \cref{thm:psi} for any $\lambda \in \C{}$?
        \item What are their analytic wavefront sets?
        \item Let $\tau \colon G \rightarrow G$ be the involution given by $\tau(g) = JgJ$, where $J$ is the orthogonal reflection in the hyperplane $x_n = 0$. Furthermore,
$$\mathfrak{g} = \mathfrak{h} \oplus \mathfrak{q}$$
with $\mathfrak{h} = \mathrm{ker}(\tau -1)$ and $\mathfrak{q} = \mathrm{ker}(\tau + 1)$. The wavefront sets of $\Psi^\lambda$ can also be written as $WF_A(\Psi^\lambda) \subset {\rm exp}(\mathfrak{q})\cdot e_n \times \Omega^\circ$ where $\Omega = \{v \in \R{n}: [v,v] = 0\}$. Can we draw similar conclusion in the context of  non compact casual symmetric spaces? 
    \end{itemize}

\end{Remark}

%--------------------------------------------------------------------------

\end{document}